\documentclass[a4paper,10pt]{article}
\usepackage[hmarginratio=1:1, bottom=3.5cm, top=2.5cm]{geometry}
\usepackage{amssymb,amsthm,amsmath,latexsym,url,graphicx, comment}
\usepackage[english]{babel} 
\usepackage[dvipsnames]{xcolor}
\usepackage[utf8]{inputenc}
\usepackage[sc,osf]{mathpazo}
\usepackage[euler-digits]{eulervm}
\usepackage[T1]{fontenc}
\usepackage{mathtools}
\usepackage{mathtools}
\usepackage{pdfpages}
\usepackage{float}
\usepackage[colorlinks=true,pdftex,unicode=true,linktocpage,bookmarksopen,hypertexnames=false]{hyperref}
\theoremstyle{dotless}

\newtheorem{theorem}{Theorem}[section]
\newtheorem{proposition}[theorem]{Proposition}
\newtheorem{lemma}[theorem]{Lemma}
\newtheorem{corollary}[theorem]{Corollary}
\theoremstyle{remark}
\newtheorem{remark}[theorem]{Remark}

\newcommand{\Aut}{\operatorname{Aut}}
\newcommand{\Inn}{\operatorname{Inn}}
\newcommand{\Out}{\operatorname{Out}}
\newcommand{\Triv}{\operatorname{Triv}}
\newcommand{\PSU}{\operatorname{PSU}}
\newcommand{\PSL}{\operatorname{PSL}}
\newcommand{\PGL}{\operatorname{PGL}}

\title{A Finite Skew Brace with Perfect Additive Group and Almost Simple Multiplicative Group}
\author{Massimiliano Di Matteo \and Ferrara Maria \footnote{The authors are members of \textit{National Group for Algebraic and Geometric Structures, and their Applications} (GNSAGA--INdAM), and members of the non-profit association \textit{Advances in Group Theory and Applications}.}}
\date{ }

\begin{document}
\maketitle

\begin{abstract} 
We construct a finite skew brace whose additive group is perfect and whose
multiplicative group is non-perfect and almost simple.  This gives an
affirmative answer to Problem 20.109 in the twenty-first edition of the
Kourovka Notebook.  The additive and multiplicative groups of our example are
\[
  \PSU_5(64)\times A_5
  \quad\text{and}\quad
  \Aut(\PSU_5(64)),
\]
respectively.  The construction combines a skew brace with additive group
\(A_5\) and multiplicative group
\((C_5\rtimes C_4)\times C_3\), the splitting of the automorphism extension
of \(\PSU_5(64)\), and a semidirect product of skew braces.\\


\end{abstract}

\section{Introduction}
A \textit{(left) skew brace} is an algebraic structure $(B,+,\circ)$, where $(B,+)$ and $(B,\circ)$ are (not necessarily abelian) groups and 
\begin{equation}\label{eq:brace-identity}
 a\circ(b+c)=a\circ b-a+a\circ c
 \qquad(a,b,c\in B).
\end{equation}
- this is known as the \textit{left skew distributivity law}. If also the right skew distributivity holds the skew brace is called \textit{two-sided}. Nevertheless, in this paper, whenever we talk about skew braces, we are referring to left skew braces, as it is the most general setting. It is relevant to note that skew braces were introduced in their present form
by Guarnieri and Vendramin~\cite{GuarnieriVendramin} and provide an algebraic
framework for non-degenerate set-theoretic solutions of the Yang--Baxter
equation.

In the twenty-first edition of the
Kourovka Notebook, the authors have proposed a list of several unsolved problems in skew brace theory~\cite{Kourovka}. In particular, Problem 20.109 has got to our attention:

\begin{quote}
\emph{Is there a finite skew brace with perfect additive group and non-perfect
almost simple multiplicative group?}
\end{quote}

Here perfectness is ordinary group-theoretic perfectness, so a perfect group is a group that coincides with its commutator subgroup. Instead, a finite group \(G\) is \textit{almost simple} if there is a finite non-abelian simple
group \(S\), which will be called the \textit{socle} of $G$, such that
\[
  S\leq G\leq\Aut(S).
\]

The aim of this paper is to construct a skew brace $B$ with the features described in Problem 20.109, and hence giving a positive answer to the question. The simple socle in our construction will be
\(\PSU_5(64)\).

\begin{theorem}\label{thm:main}
There exists a finite skew brace \(\mathcal B=(\mathcal B,+,\circ)\) such that
\[
 (\mathcal B,+)\simeq \PSU_5(64)\times A_5
\]
and
\[
 (\mathcal B,\circ)\simeq \Aut(\PSU_5(64)).
\]
The additive group is perfect, while the multiplicative group is non-perfect
and almost simple with socle \(\PSU_5(64)\).
\end{theorem}

It will be important to recall some important properties about $\Aut(\PSU_5(64))$. Of course, since $\PSU_5(64)$ is a non-abelian simple group, then $\Inn(\PSU_5(64))\simeq \PSU_5(64)$; instead its outer automorphism group is
\[
  \Out(\PSU_5(64))\simeq (C_5\rtimes C_4)\times C_3,
\]
which is a solvable group and can occur as the moltiplicative group of a skew brace whose additive group is isomorphic to $A_5$.\\
The most significant fact for our purposes is that, for $\PSU_5(64)$, the canonical short exact sequence
$$0\rightarrow \Inn(\PSU_5(64))\rightarrow\Aut(\PSU_5(64))\rightarrow  \Out(\PSU_5(64))\rightarrow 0$$
is split and so $\Aut(\PSU_5(64))\simeq \Inn(\PSU_5(64))\rtimes\Out(\PSU_5(64))$. This result is crucial, as it will give us the possibility to use a general semidirect product construction for skew braces. 
  

Throughout the paper every skew brace is denoted by \((B,+,\circ)\). Thus \(+\) and \(\circ\) are always, respectively, its additive and multiplicative
operations. Juxtaposition is used only for products or composition in auxiliary ordinary groups, never for an operation of a skew brace. 

\section{Two skew-brace constructions}

Let $(B,+,\circ)$ a skew brace, we will denote $-a$ the additive inverse of any element $a$ in $B$. Moreover, the \textit{lambda function} of $B$ is the following group homomorphism:
\[
    a\in (B,\circ)\longmapsto \lambda_a\in \operatorname{Aut}(B,+).
\]
where $\lambda_a(b)=-a+a\circ b$ for any $a$ and $b$ in $B$.

\subsection{The auxiliary brace of order sixty}

Let
\[
 F_{20}=C_5\rtimes C_4
 =\langle u,v\mid u^5=v^4=1,\ vuv^{-1}=u^2\rangle.
\]

\begin{proposition}\label{prop:auxiliary}
There exists a skew brace \(C=(C,+,\circ)\) satisfying
\[
 (C,+)\simeq A_5,
 \qquad
 (C,\circ)\simeq F_{20}\times C_3.
\]
\end{proposition}

\begin{proof}
We recall the specialization to \(q=5\) of the construction in
\cite[Proposition 2.7 and the proof of Theorem 1.3]{TsangSolvable}.
Put
\[
 L=\PGL_2(5)\simeq S_5,
 \qquad
 N=\PSL_2(5)\simeq A_5.
\]
A Singer cycle \(X\) of \(L\) is cyclic of order six, while the stabilizer
\(Y\) of a point of the projective line has order twenty and is isomorphic to
\(F_{20}\).  The Singer cycle is regular on the six points, so
\[
 L=XY,
 \qquad X\cap Y=1.
\]
Moreover,
\[
 X_0=X\cap N\simeq C_3,
 \qquad
 Y_0=Y\cap N\simeq D_{10},
\]
and both \(X\) and \(Y\) map onto \(L/N\simeq C_2\).

The cited construction gives us a skew brace with additive group \(N\) and
multiplicative group \(X_0\rtimes_\alpha Y\).  The action \(\alpha\) factors
through \(Y/Y_0\simeq C_2\) and is induced by conjugation by the complement
of \(X_0\) in \(X\).  Since \(X\simeq C_6\) is abelian, this action is trivial.
Consequently
\[
 X_0\rtimes_\alpha Y\simeq C_3\times F_{20},
\]
which proves the claim.
\end{proof}

\subsection{A semidirect product}

If \((S,+)\) is a group, let \(\Triv(S)\) denote the trivial skew brace on
\(S\), in which the two operations coincide.  Its skew-brace automorphism
group is \(\Aut(S,+)\).

\begin{lemma}[Semidirect product]\label{lem:semidirect}
Let \(C=(C,+,\circ)\) be a skew brace, let \((S,+)\) be a group, and let
\[
 \theta:(C,\circ)\longrightarrow\Aut(S,+)
\]
be a group homomorphism.  On \(S\times C\), define
\begin{align}
 (s,x)+(t,y)&=(s+t,x+y),                                      \label{eq:sd-add}\\
 (s,x)\circ(t,y)&=(s+\theta_x(t),x\circ y).                  \label{eq:sd-circle}
\end{align}
Then \((S\times C,+,\circ)\) is a skew brace.  Its additive and
multiplicative groups are
\[
 (S,+)\times(C,+)
 \quad\text{and}\quad
 (S,+)\rtimes_\theta(C,\circ),
\]
respectively.
\end{lemma}

\begin{proof}
Equation~\eqref{eq:sd-circle} is the usual semidirect-product group law.  The
lambda map associated with~\eqref{eq:sd-add}--\eqref{eq:sd-circle} is
\[
 \lambda_{(s,x)}(t,y)=\bigl(\theta_x(t),\lambda_x^C(y)\bigr).
\]
Both components are automorphisms of the corresponding direct factors of the
additive group.  Hence every \(\lambda_{(s,x)}\) is an additive automorphism,
which proves~\eqref{eq:brace-identity}.
\end{proof}

Note that this is the usual brace-theoeretic semidirect product, introduced in \cite[Section 2]{SmoktunowiczVendramin}, between \(\Triv(S)\) and \(C\). 

\section{The automorphism group of \texorpdfstring{$\PSU_5(64)$}{PSU5(64)}}
\label{sec:automorphisms}

We record the group-theoretic input needed in the proof.

\begin{proposition}\label{prop:complement}
Let \(S=\PSU_5(64)\).  Then
\[
 \Aut(S)=S\rtimes H
 \qquad\text{with}\qquad
 H\simeq F_{20}\times C_3.
\]
In particular, \(H\cap\Inn(S)=1\).
\end{proposition}

\begin{proof}
Write \(q=64=2^6\) and
\[
 d=\gcd(5,q+1)=5.
\]
For a projective special unitary group of degree at least three, the outer
automorphism group has the form
\[
 \Out(\PSU_n(p^m))
 =\langle\delta\rangle\rtimes\langle\varphi\rangle,
 \qquad
 |\delta|=\gcd(n,p^m+1),\quad |\varphi|=2m,
\]
where \(\delta^\varphi=\delta^p\); see
\cite[Section 4]{CostantiniLucchiniNemmi}.  In the present case,
\[
 \Out(S)\simeq C_5\rtimes C_{12},
\]
and a generator of \(C_{12}\) acts on \(C_5\) by the power map with exponent
two.  This action has order four and kernel of order three.  Hence
\begin{equation}\label{eq:outer-structure}
 \Out(S)\simeq (C_5\rtimes C_4)\times C_3
 \simeq F_{20}\times C_3.
\end{equation}

It remains to lift the outer group to a complement.  The splitting criterion
of Lucchini, Menegazzo and Morigi~\cite{LucchiniMenegazzoMorigiSplit} says
that for a twisted group, other than the listed orthogonal and Tits
exceptions, the full automorphism extension splits precisely when
\[
 \gcd\!\left(\frac{q+1}{d},d,m\right)=1.
\]
For \(S=\PSU_5(64)\), $q=64$, $d=\gcd(5,65)=5$ and $m=6$, so this is
\[
 \gcd(13,5,6)=1.
\]
Therefore \(\Aut(S)\) splits over \(S=\Inn(S)\).  A complement \(H\) is isomorphic
to \(\Out(S)\), and~\eqref{eq:outer-structure} finishes the proof.
\end{proof}

\section{Proof of the main theorem}

Let \(C=(C,+,\circ)\) be the skew brace supplied by
Proposition~\ref{prop:auxiliary},  \(S=\PSU_5(64)\), written additively, and choose the complement \(H\leq\Aut(S,+)\) from
Proposition~\ref{prop:complement}, such that $(C,\circ)\simeq F_{20}\times C_3\simeq H$.

\begin{proof}[Proof of Theorem~\ref{thm:main}]
We may choose an isomorphism
\[
 \theta:(C,\circ)\longrightarrow H\leq\Aut(S,+).
\]
and apply Lemma~\ref{lem:semidirect} to \(\Triv(S)\) and \(C\), and denote the
resulting skew brace by
\[
 \mathcal B=\Triv(S)\rtimes_\theta C.
\]
By construction,
\begin{equation}\label{eq:underlying-groups}
 (\mathcal B,+)\simeq S\times A_5,
 \qquad
 (\mathcal B,\circ)\simeq S\rtimes H\simeq\Aut(S).
\end{equation}
Both \(S\) and \(A_5\) are non-abelian simple groups and therefore perfect.
It follows from~\eqref{eq:underlying-groups} that
\[
 [S\times A_5,S\times A_5]=S\times A_5;
\]
hence the additive group of \(\mathcal B\) is perfect.

The multiplicative group is \(\Aut(S)\), so it is almost simple with socle
\(S\).  Finally,
\[
 (\mathcal B,\circ)/S\simeq H\simeq F_{20}\times C_3.
\]
The derived subgroup of \(F_{20}\) is \(C_5\), and therefore
\[
 H_{\mathrm{ab}}\simeq C_4\times C_3\simeq C_{12}.
\]
Thus \(H\) is not perfect.  Since every quotient of a perfect group is
perfect, \((\mathcal B,\circ)\) cannot be perfect.  This proves every
assertion.
\end{proof}

\begin{remark}
The construction gives a skew left brace; no two-sidedness is asserted.
Indeed, a finite two-sided skew brace with perfect additive group has perfect
multiplicative group~\cite[Corollary 1]{DamelePerfect}.
\end{remark}

\newpage

\begin{flushleft}
\rule{8cm}{0.4pt}\\
\end{flushleft}

{
\sloppy
\noindent
Massimiliano di Matteo

\noindent
Dipartimento di Matematica e Fisica

\noindent
Università degli Studi della Campania  ``Luigi Vanvitelli''

\noindent
viale Lincoln 5, Caserta (Italy)

\noindent
e-mail: massimiliano.dimatteo@unicampania.it 
}

\bigskip
\bigskip

{
\sloppy
\noindent
Maria Ferrara

\noindent
Dipartimento di Ingegneria

\noindent
Facoltà di Ingegneria e Informatica

\noindent
Università Pegaso\\
Centro Direzionale Isola F2 - Napoli (Italy)

\noindent
e-mail: maria.ferrara1@unipegaso.it 

}


{\small
\begin{thebibliography}{99}

\bibitem{CostantiniLucchiniNemmi}
M.~Costantini, A.~Lucchini and D.~Nemmi,
\emph{Abelian supplements in almost simple groups},
Forum Math. Sigma \textbf{13} (2025), e14, 1--33.
\href{https://doi.org/10.1017/fms.2024.160}{doi:10.1017/fms.2024.160}.

\bibitem{DamelePerfect}
M.~Damele,
\emph{On finite perfect two-sided skew braces},
preprint, 2026.
\href{https://arxiv.org/abs/2605.22302}{arXiv:2605.22302}.

\bibitem{GuarnieriVendramin}
L.~Guarnieri and L.~Vendramin,
\emph{Skew braces and the Yang--Baxter equation},
Math. Comp. \textbf{86} (2017), no.~307, 2519--2534.
\href{https://doi.org/10.1090/mcom/3161}{doi:10.1090/mcom/3161}.

\bibitem{Kourovka}
E.~I. Khukhro and V.~D. Mazurov (eds.),
\emph{Unsolved Problems in Group Theory: The Kourovka Notebook},
21st ed., 2026.
\href{https://arxiv.org/abs/1401.0300}{arXiv:1401.0300}.

\bibitem{LucchiniMenegazzoMorigiSplit}
A.~Lucchini, F.~Menegazzo and M.~Morigi,
\emph{On the existence of a complement for a finite simple group in its
automorphism group},
Illinois J. Math. \textbf{47} (2003), no.~1--2, 395--418.
\href{https://doi.org/10.1215/ijm/1258488162}{doi:10.1215/ijm/1258488162}.

\bibitem{SmoktunowiczVendramin}
A.~Smoktunowicz and L.~Vendramin,
\emph{On skew braces}, with an appendix by N.~Byott and L.~Vendramin,
J. Combin. Algebra \textbf{2} (2018), no.~1, 47--86.
\href{https://doi.org/10.4171/JCA/2-1-3}{doi:10.4171/JCA/2-1-3}.

\bibitem{TsangSolvable}
C.~S.~Y.~Tsang,
\emph{Non-abelian simple groups which occur as the type of a Hopf--Galois
structure on a solvable extension},
Bull. Lond. Math. Soc. \textbf{55} (2023), no.~5, 2324--2340.
\href{https://doi.org/10.1112/blms.12860}{doi:10.1112/blms.12860}.

\end{thebibliography}
}
\end{document}